\documentclass[a4paper,11pt]{article}

\usepackage[T2A,T1]{fontenc}
\usepackage[utf8]{inputenc}
\usepackage[russian,english]{babel}

\usepackage{latexsym,graphics,graphicx,color,verbatim}
\usepackage{amsmath,amsfonts,amssymb,amsthm}
\newtheorem{thm}{Theorem}[section]
\newtheorem{lem}[thm]{Lemma}
\newtheorem{exple}[thm]{Example}

\newtheorem{cor}[thm]{Corollary}
\newtheorem{prop}[thm]{Proposition}

\newtheorem{defn}[thm]{Definition}
\input epsf

\title{A Quixotic Proof of Fermat's Two Squares Theorem for Prime Numbers}
\author{Roland Bacher}

\begin{document}
\maketitle

\begin{abstract} Every odd prime number $p$ can be written in
  exactly $(p+1)/2$ ways as a sum $ab+cd$ of two ordered products $ab$
  and $cd$ such that
  $\min(a,b)>\max(c,d)$. An easy corollary is a proof of Fermat's Theorem
  expressing primes
  in $1+4\mathbb N$ as sums of two squares 
  \footnote{
 Math. class: Primary:
 11A41.
 Secondary: 11H06.
}.
\end{abstract}



\section{Introduction}

\begin{thm}\label{thmmain} For every odd prime number $p$ there exist
  exactly $(p+1)/2$ sequences $(a,b,c,d)$ of four elements
  in the set $\mathbb N=\{0,1,2,\ldots\}$ of non-negative integers
  such that $p=a b+c d$ and $\min(a,b)>\max(c,d)$.
\end{thm}

As a consequence we obtain a new proof of an old result first observed by
Albert Girard (1595-1632) (who wrote moreover that the set of integers which are sums
of two squares contains $2$ and is closed under product)
around 1625. A refined statement including multiplicities was written in
1640 by Pierre Fermat (1607-1665) in a letter addressed to Marin Mersenne
(1588-1648).
The old rascal did not want to spoil his margins
and left the proof (contained in a series of letters and publications
dated around 1750) to Leonhard Euler (1707-1783)
who had no such qualms. More historical
details can for example
be found in the entry ``Fermat's theorem on sums of two squares''
of \cite{WikipFerm}. 

\begin{cor}\label{oldhat} Every prime number in $1+4\mathbb N$ is a sum
  of two squares.
\end{cor}

\begin{proof}[Proof of Corollary \ref{oldhat}] If $p$ is a prime number congruent to
  $1\pmod 4$, the number $(p+1)/2$ of solutions $(a,b,c,d)$ defined by
  Theorem \ref{thmmain} is odd. The involution $(a,b,c,d)\longmapsto
  (b,a,d,c)$ has therefore a fixed point $(a,a,c,c)$ expressing $p$ as a sum of two squares.
  \end{proof}

  Nowadays a venerable old hat, Corollary \ref{oldhat} has of course already
  quite a few proofs. Some are described in \cite{WikipFerm}.   
  The author enjoyed the presentation of a few \lq\lq elementary\rq\rq proofs
  given in \cite{Els}.

  Zagier (based on unpublished notes of Heath-Brown, \cite{HB}) published a one sentence proof based on fixed points in \cite{Za}. A. Spivak gave an
  elementary geometric
  interpretation of Zagier's proof, see \cite{Sp}. A nice variation on
  Zagier's proof was given by Dolan in \cite{Do}. An interesting
  discussion on Zagier's proof and variations can be found at
  \cite{Mo} which describes moreover A. Spivak's proof in an answer.

Grace gave a very elegant constructive proof, see \cite{Gra} (essentially
equivalent to the fourth proof of Theorem 366 in \cite{HW}) which
we recall in Section \ref{sectGrace} for the convenience of the reader.

  Christopher, see \cite{Chr}, gave a proof based on the existence of a fixed point of an involution
  acting on suitable partitions with parts of exactly two different sizes
  (amounting essentially to solutions of $p=ab+cd$ without
  requirements of inequalities).


  The set $\mathcal S_p$ of solutions defined by Theorem \ref{thmmain}
  is invariant under the action of Klein's Vierergruppe $\mathbb V$ (isomorphic to the
  $2$-dimensional vector space over the field of two elements or, equivalently,
  isomorphic to the unique non-cyclic group of four elements) with
  non-trivial elements acting by
  $$(a,b,c,d)\longmapsto (b,a,c,d),(a,b,d,c),(b,a,d,c)$$
  (i.e. by exchanging either the first two elements, or the last two elements,
  or the first two and the last two elements).
  The following tables list all $\mathbb V$-orbits represented by
  elements $(a,b,c,d)$ with
  $a,b,c,d$ decreasing together with the orbit-sizes $\sharp(\mathcal O)$
  occurring in the sets
  $\mathcal S_{29},\mathcal S_{31},\mathcal S_{37}$:
  $$\begin{array}{ccc}
      \begin{array}{cccc|c}
        a&b&c&d&\sharp(\mathcal O)\\
        \hline
29&1&0&0&2\\
14&2&1&1&2\\
7&4&1&1&2\\
9&3&2&1&4\\
5&5&4&1&2\\
5&5&2&2&1\\
5&4&3&3&2\\
\hline
&&&&15
      \end{array}\quad&
       \begin{array}{cccc|c}
        a&b&c&d&\sharp(\mathcal O)\\
        \hline
31&1&0&0&2\\
15&2&1&1&2\\
10&3&1&1&2\\
6&5&1&1&2\\
7&4&3&1&4\\
9&3&2&2&2\\
5&5&3&2&2\\
\hline
&&&&16
       \end{array}\quad&
       \begin{array}{cccc|c}
        a&b&c&d&\sharp(\mathcal O)\\
        \hline
37&1&0&0&2\\
18&2&1&1&2\\
12&3&1&1&2\\
9&4&1&1&2\\
6&6&1&1&1\\         
7&5&2&1&4\\
11&3&2&2&2\\
7&4&3&3&2\\
5&5&4&3&2\\
\hline
&&&&19
       \end{array}     
    \end{array}
    $$
      \begin{figure}[h]
  \epsfysize=3.8cm
\center{\epsfbox{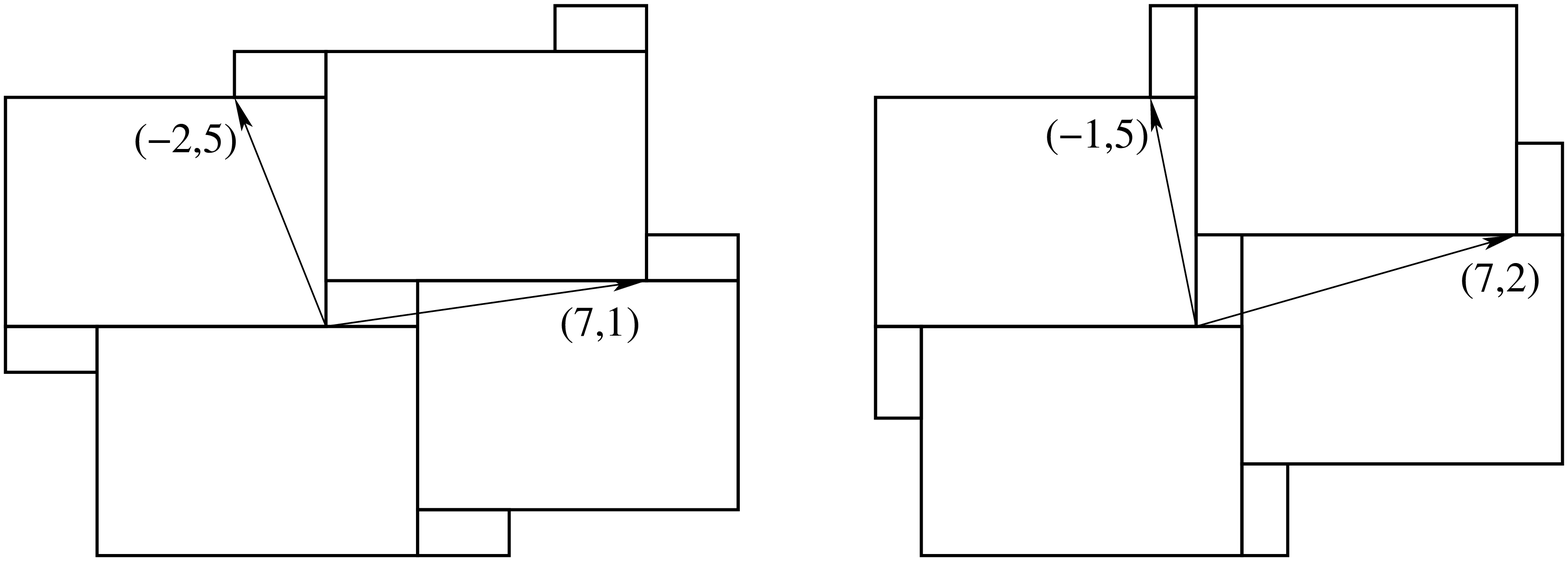}}
\caption{Tilings associated to $7\cdot 5+1\cdot 2$ and $7\cdot 5+2\cdot 1$.}\label{figtil}
\end{figure}
    Establishing complete lists $\mathcal S_p$ of solutions
    for small primes is a rather pleasant pastime
    and rates among the author's more confessable procrastinations.

    A solution $p=ab+cd$ in $\mathcal S_p$ can be visualized as a lattice-tiling
    with a fundamental domain given by the union of two rectangles of size $a\times b$ and $d\times c$, aligned as in Figure
    \ref{figtil}. The resulting tiling is invariant by translations in
    $\mathbb Z(a,c)+\mathbb Z(-d,b)$. Klein's Vierergruppe
    $\mathbb V$ acts on the set of
    all such tilings by quarter-turns on rectangles.
    Tilings associated to $\mathbb V$-invariant solutions
    correspond to the case
    where both rectangles are squares.

    A rough sketch for proving Theorem \ref{thmmain} goes along the following lines: Every solution
    $p=ab+cd$ in $\mathcal S_p$ can be encoded by two
    vectors $u=(a,c),v=(-d,b)$ generating a sublattice $\mathbb Z u+\mathbb Zv$
    of index $p$ in $\mathbb Z^2$, see above and Figure \ref{figtil}.
    It is therefore enough to understand the number of
    solutions encoded by every sublattice of index $p$ in $\mathbb Z^2$.
    Sublattices of index $p$ in $\mathbb Z^2$
    are in one-to-one correspondence with all $p+1$ elements of the projective line $\mathbb F_p\cup\{\infty\}$ over the finite field $\mathbb F_p$. An element $\mu$ encoding the slope $\mu=\frac{b}{a}$ of $[a:b]$ (using the obvious convention for $\mu=\infty$) defines the sublattice
    $\Lambda_\mu(p)=\{(x,y)\in\mathbb Z^2\ \vert\ ax+by\equiv 0\pmod p\}$
    of index $p$ in $\mathbb Z^2$. The two lattices $\Lambda_0(p)=\mathbb Z(p,0)+\mathbb Z(0,1)$ and $\Lambda_\infty(p)=\mathbb Z(1,0)+\mathbb Z(0,p)$
    with singular slopes $0,\infty\not\in \mathbb F_p^*$ give rise to the
    two degenerate solutions $p\cdot 1+0\cdot 0$ and $1\cdot p+0\cdot 0$.
    All other solutions $p=ab+cd$ correspond to sublattices of index $p$ in
    $\mathbb Z^2$ generated by $u=(a,c)$ in the open cone delimited by $y=0$ and $x=y$ (opening up in E-NE directions) and by $v=(-d,c)$ in the open cone
    delimited by $x=0$ and $y=-x$ (opening up in N-NW directions).
    The two lattices $\Lambda_1(p)$ and $\Lambda_{-1}(p)$ with
    self-inverse slopes $1$ and $-1$ have no such bases and thus do
    not correspond to a solution. Exactly one lattice in every
    pair of distinct lattices $\Lambda_\mu(p),\Lambda_{\mu^{-1}}(p)$
    with mutually
    inverse slopes $\mu,\mu^{-1}\in\mathbb F_p^*\setminus\{1,-1\}$
    has bases with generators in the two open E-NE and N-NW cones. We show then
    that exactly one of these bases corresponds to a
    solution in $\mathcal S_p$. Theorem \ref{thmmain} follows now easily.

    Colouring the set $\{(x,y)\ \vert xy(x-y)(x+y)>0\}$ consisting
    of the four open E-NE, N-NW, W-SW and S-SE cones and their opposites
    in black we get a picture of the four sails of an old windmill,
    see Figure \ref{figwm}. We prove therefore Theorem \ref{thmmain}
    by following in Don Quixote's heroic footsteps
    (see the beginning of Chapter 8 in \cite{C}).
    Cervantes forgot of course the explicit statement of
    Theorem \ref{thmmain} and botched the proof by
    sweeping all those bloody details under the rug.


    The author's serendipitous encounter with Don Quixote happened as
    follows: Euclid's algorithm applied to square-matrices of size $2$
(replacing iteratively a row/column by itself minus the other row/column)
with entries in the set $\{0,1,\ldots\}$ of non-negative integers
 yields a set
\begin{align*}
&\left\lbrace\left(\begin{array}{cc}a&b\\c&d\end{array}\right)\vert ad-bc=n,\
                                       \min(a,d)>\max(b,c),\ a,b,c,d\in\{0,1,2,\ldots\}\right\rbrace
\end{align*}
of
\begin{align}\label{formirrmat}
  &\sum_{d\vert n,\ d^2\geq n}\left(d+1-\frac{n}{d}\right)
\end{align}
\lq\lq irreducible\rq\rq matrices of given determinant $n\geq 1$, see
\cite{MO405035}.

A sign-change made out of curiosity 
in a naive program checking Formula (\ref{formirrmat})
for small values of $n$ suggested Theorem \ref{thmmain}.

A final additional Section links solutions
occuring in Theorem \ref{thmmain} with geometric properties of
the corresponding lattices.
    


    \section{A few reminders on lattices in $\mathbb R^2$}

    This short Section contains a few elementary and
    well-known results on lattices in
    $\mathbb R^2$, recalled for the convenience of the reader. 

    A lattice denotes henceforth a discrete additive group
    $\mathbb Ze+\mathbb Zf$ generated by a an arbitrary basis $e,f$
of the Euclidean vector-space $\mathbb R^2$
endowed with the standard scalar product $\langle (u_x,u_y),(v_x,v_y)\rangle
=u_xv_x+u_yv_y$
of two vectors in $\mathbb R^2$. We will mainly work with
    sublattices of the integral lattice $\mathbb Z^2$ in $\mathbb R^2$.

    A \emph{minimal element} of a lattice $\Lambda$ is a shortest
    element in $\Lambda\setminus\{(0,0)\}$.

    An element $v$ of $\Lambda$ is \emph{primitive} if it is not contained in
    $k\Lambda$ for some natural integer $k>1$.
    
   A \emph{basis}
    of a lattice $\Lambda$ of rank (or dimension) $2$
    is a set $e,f$ of two elements in $\Lambda$ such that
    $\Lambda=\mathbb Z e+\mathbb Z f$.

The following result is a special case of Pick's Theorem\footnote{Pick's theorem gives the area $\frac{1}{2}b+i-1$ of a closed lattice polygon $P$
    (with vertices in $\mathbb Z^2$)
    containing $b$ lattice points $\partial P\cap \mathbb Z^2$
    in its boundary and $i$ lattice points in its interior.}:

  \begin{lem}\label{lemlatticebasis} Two linearly independent
    elements $e,f$ of a
    $2$-dimensional lattice $\Lambda$ 
    form a basis of the lattice $\Lambda$
    if and only if the triangle with vertices $(0,0),e,f$
    contains no other elements of $\Lambda$.
\end{lem}

\begin{proof} The parallelogram $\mathcal P$ with vertices
  $(0,0),e,f,e+f$ is a fundamental domain for the sublattice
  $\mathbb Z e+\mathbb Z f$ of $\Lambda$ generated by $e$ and $f$.
  The two elements $e,f$ generate therefore $\Lambda$ if and only if
  $\Lambda$ intersects $\mathcal P$ only in its four vertices.
  
  Since $\mathcal P$ and $\Lambda$ are invariant under the affine involution
  $x\longmapsto -x+e+f$ exchanging the two triangles with vertices $(0,0),e,f$
  and $e,f,e+f$, the parallelogram $\mathcal P$ intersects
  $\Lambda$ exactly in its vertices if and only if the triangle
  defined by $(0,0),e,f$ intersects $\Lambda$ exactly in its vertices.
\end{proof}

\begin{prop}\label{propsubl} A lattice $M$ in
  $\mathbb R^2$ has exactly $p+1$ different sublattices
  of index a prime number $p$. These sublattices are in one-to-one correspondence
  with the set of lines in $M/pM$ representing all elements of the
  projective line over the finite field $\mathbb F_p$.
\end{prop}

\begin{proof} A sublattice $\Lambda$ of prime index $p$ in a $2$-dimensional
  lattice $M$ gives
  rise to a quotient group $M/\Lambda$ isomorphic to the additive
  group $\mathbb Z/p\mathbb Z$. Since $pM$ is
  contained in the kernel of the quotient
  map $M\longmapsto M/\Lambda$, subgroups of index $p$ in $M$ are in bijection
  with kernels of linear surjections from the $2$-dimensional vector-space
  $M/pM$ over $\mathbb F_p$ onto $\mathbb F_p$ considered as a
  $1$-dimensional vector-space.
  The set of all subgroups of index $p$ in $M$ is thus in bijection
  with the set of $1$-dimensional subspaces in $\mathbb F_p^2$
  representing all possible kernels.
  Such subspaces represent all $p+1$ points of the projective line
  over $\mathbb F_p$.
\end{proof}

The Euclidean algorithm computes the positive generator of the cyclic subgroup
$\mathbb Z a+\mathbb Z b$ of $\mathbb Z$
generated by two integers $a$ and $b$.
\emph{Gau\ss ian lattice reduction} does essentially the same for lattices
in the Euclidean space $\mathbb R^2$: Given
two linearly independent\footnote{Gau\ss ian lattice reduction applied
  to two linearly dependent vectors boils down 
to the Euclidean algorithm.} vectors
$e,f$ in $\mathbb R^2$, the Gau\ss ian algorithm produces
a \emph{reduced basis} $r,s$
of the lattice $\mathbb Z e+\mathbb Z f=\mathbb Z r+\mathbb Z s$
defining two (not necessarily unique) distinct shortest pairs $\pm r,\pm s$
of primitive vectors. The Gau\ss ian algorithm starts with a basis $e,f$
of $\Lambda=\mathbb Z e+\mathbb Z f$ and 
iterates the following two steps until stabilization:
\begin{itemize}
\item{} Exchange $e$ and $f$ if $f$ is strictly longer than $e$.
\item{} Replace $e$ by $e+kf$ if $e+kf$ (for $k$ an integer)
  is strictly shorter than $e$
  (the optimal choice for $k$ is given by $k\in\mathbb Z$ such that
  $\left\vert k+\frac{\langle f,e\rangle}{\langle f,f\rangle}\right\vert$ is at most equal to $\frac{1}{2}$).
\end{itemize}

Finally, the following (obvious) result will also be needed a few times:

\begin{prop}\label{propindex} The sublattice $\mathbb Z(\alpha e+\beta f)
  +\mathbb Z(\gamma e+\delta f)$ of a lattice $\Lambda=\mathbb Z e+\mathbb Z f$
  generated by two linearly independent vectors
  $\alpha e+\beta f$ and
  $\gamma e+\delta f$ (with $\alpha,\beta,\gamma,\delta$ in $\mathbb Z$)
  has index $\vert \alpha\delta-\beta\gamma\vert$
  in $\Lambda$.
\end{prop}

\begin{proof} The result is obvious if $\alpha\beta\gamma\delta=0$.
  The general case can be reduced by elementary operations on the generators
  $u=\alpha e+\beta f$ and
  $v=\gamma e+\delta f$ to the obvious case.
\end{proof}

\section{Grace's proof}\label{sectGrace}
      
For the convenience of the reader, we recall
Grace's proof (of Corollary \ref{oldhat}),
as given in \cite{HW} (see the fourth proof of Theorem 366 in \cite{HW} or
\cite{Gra})\footnote{The only difference is the fact that Grace
  admits the existence of square roots of $-1$ modulo primes congruent to
  $1\pmod 4$. The author succumbed to the siren song of a well-known
  explicit construction (based on Wilson's Theorem) for such a square-root.}.

\begin{proof} Given an odd prime number $p$ we have
$\left(\left(\frac{p-1}{2}\right)!\right)^2(-1)^{(p-1)/2}\equiv (p-1)!\pmod p$
which equals $-1\pmod p$ by Wilson's Theorem. For $p$ a prime
number congruent to $1\pmod 4$,
the integer $\iota=\left(\frac{p-1}{2}\right)!$ (and its opposite)
is a square root of $-1$ in the finite field $\mathbb F_p$.
The kernel of the homomorphism
$\mathbb Z^2\ni(x,y)\longmapsto x+\iota y\pmod p$
is a sublattice $\Lambda$ of index $p$ in $\mathbb Z^2$.
Since $\iota(x+\iota y)\equiv -y+\iota x\pmod p$,
we have $(x,y)\in \Lambda$ if and only if $(-y,x)\in\Lambda$. The lattice
$\Lambda$ is therefore invariant under quarter-turns 
(rotations of order $4$ by $\pm 90$ degrees).
Let $(a,b)$ be a non-zero element of minimal
length in $\Lambda$. Invariance under quarter-turns of
$\Lambda$ implies that $(-b,a)$ is also an element of $\Lambda$.
Length-minimality of $(a,b)$ and $(-b,a)$ implies that the
triangle with vertices $(0,0),(a,b),(-b,a)$ contains no other
element of $\Lambda$. Lemma \ref{lemlatticebasis}
shows that the two vectors $(a,b)$
and $(-b,a)$ generate $\Lambda$. Proposition \ref{propindex}
implies now that $a^2+b^2$ is equal to the index
$p$ of the sublattice $\Lambda$ in $\mathbb Z^2$.
\end{proof}

Grace's proof is effective: Given a prime $p\equiv 1\pmod 4$, we can 
use quadratic reciprocity for finding a non-square $n$ modulo $p$.
We obtain a square root $\iota$ of $-1$ in $\mathbb F_p$ by
computing $\iota\equiv n^{(p-1)/4}\pmod p$ using fast exponentiation.
We get now a solution by considering an element $(a,b)$  
of a reduced basis (obtained by 
Gau\ss ian lattice reduction) of the lattice generated by $(p,0)$ and
$(-\iota,1)$.

\section{Interlacedness}

\begin{defn} Two unordered
bases $f_1,f_2$ and $g_1,g_2$ of $\mathbb R^2$ are \emph{interlaced}
if $f_1,f_2,g_1,g_2$ represent four distinct points of the real
projective line such that the two projective points represented by
$f_1,f_2$ separate the two projective points represented by $g_1,g_2$.
\end{defn}

Interlacedness
can be defined equivalently as follows: Colour the two lines $\mathbb R f_i$
associated to the first basis $f_1,f_2$ fuchsia
and colour the two lines $\mathbb R g_i$ green. Then the set
$\mathbb R f_1\cup \mathbb R f_2\cup\mathbb R g_1\cup \mathbb R g_2$
should contain four different lines and colours should alternate.

\begin{exple} The standard basis $(1,0),(0,1)$ of $\mathbb R^2$ is interlaced
with the basis $(-1,2),(6,1)$. The standard basis $(1,0),(0,1)$ is however not
interlaced with the basis $(2,3),(-1,-1)$.
\end{exple}

\begin{lem}\label{lemnoncrossing} Two bases $f_1,f_2$ and $g_1,g_2$
  of a $2$-dimensional 
  lattice $\Lambda=\mathbb Z f_1+\mathbb Z f_2=\mathbb Z g_1+\mathbb Z g_2$
  are never interlaced.
\end{lem}

\begin{proof} Suppose that a lattice $\Lambda=\mathbb Z f_1+\mathbb Z f_2
  =\mathbb Z g_1+\mathbb Z g_2$  has two bases $f_1,f_2$ and $g_1,g_2$
  which are interlaced.
  After replacing $f_1$ and $f_2$ by their negatives if necessary and
  perhaps after exchanging $f_1$ with
  $f_2$ we can suppose that $f_1$ belongs to the open cone
  spanned by $g_1$ and $-g_2$ and $f_2$ belongs to the open cone
  spanned by $g_1$ and $g_2$. Since we are working with bases of a lattice $\Lambda$,
  there exist strictly positive
  integers $\alpha,\beta,\gamma,\delta$ such that
  $f_1=\alpha g_1-\beta g_2$
  and $f_2=\gamma g_1+\delta g_2$ which can be rewritten as
  $\left(\begin{array}{c}f_1\\f_2\end{array}\right)=\left(\begin{array}{cc}\alpha&-\beta\\ \gamma&\delta\end{array}\right)\left(\begin{array}{c}g_1\\g_2\end{array}\right)$. Proposition \ref{propindex} applied to the inequality $\det\left(\begin{array}{cc}\alpha&-\beta\\ \gamma&\delta\end{array}\right)=\alpha\delta+\beta\gamma\geq 2$ implies now that $\mathbb Z f_1+\mathbb Z f_2$ is a strict sublattice of
  index at least $2$ in $\Lambda=\mathbb Z g_1+\mathbb Z g_2$.
\end{proof}


    \section{Proof of Theorem \ref{thmmain}}\label{sectproof}


We consider the eight open cones
of $\mathbb R^2$ delimited by the four lines $x=0$, $y=0$, $x=y$ and $x=-y$.
  We call these eight open cones \emph{windmill cones}
  and we colour them alternately black
  and white, starting with a black E-NE windmill cone
  $\{(x,y)\ \vert\ 0<y<x\}$ (using the conventions
  of wind-roses), as illustrated in Figure \ref{figwm}.
  \begin{figure}[h]
  \epsfysize=3cm
\center{\epsfbox{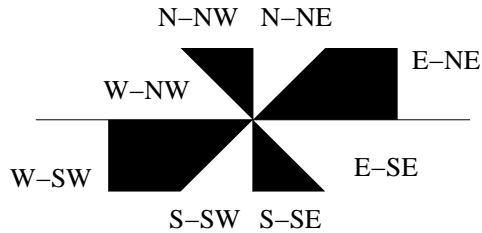}}
\caption{The four black windmill cones E-NE, N-NW, W-SW, S-SE and the four white windmill cones N-NE, W-NW, S-SW, E-SE.}\label{figwm}
\end{figure}

  A basis $e,f$ of $\mathbb R^2$ is a \emph{black windmill basis} if $e$ and $f$ are contained in the open upper half-plane and
  if one element in $\{e,f\}$ belongs to the open black E-NE windmill cone and the other element in $\{e,f\}$ belongs to
  the open black N-NW windmill cone. Similarly, a \emph{white windmill} basis contains
  an element in the open white N-NE windmill cone and an element in the
  open white W-NW windmill cone.

  A $2$-dimensional lattice $\Lambda$ in $\mathbb R^2$
  has a black (respectively white)
  windmill basis if $\Lambda=\mathbb Z e+\mathbb Z f$ is generated
  by a black (respectively white) windmill basis $e,f$.

\begin{figure}[h]
  \epsfysize=6cm
\center{\epsfbox{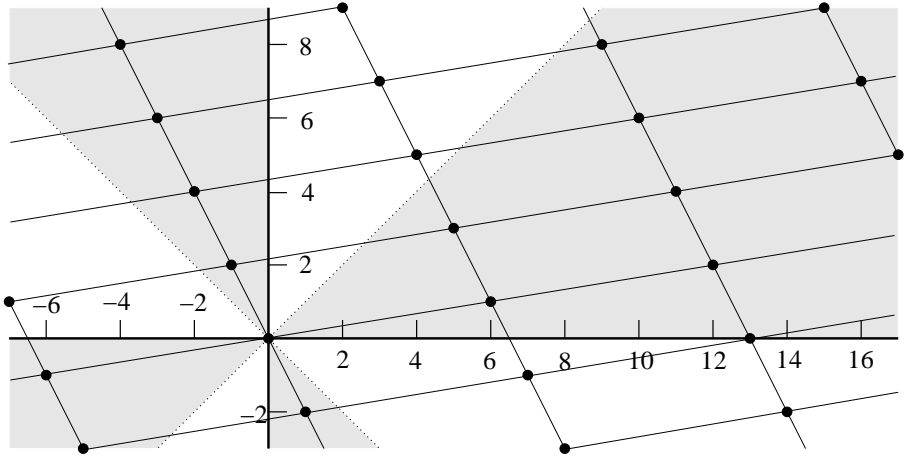}}
\caption{The lattice $\Lambda=\mathbb Z(-1,2)+\mathbb Z(6,1)$.}\label{figrunexple}
\end{figure}

  We illustrate the notion of windmill bases with the following example,
  henceforth used as a running example: Consider the lattice
  \begin{align}\label{runningexple}
    \Lambda&=\{(x,y)\in\mathbb Z^2\ \vert\ x+7y\equiv 0\pmod{13}\}
  \end{align}
  depicted in Figure \ref{figrunexple} with shaded
  black windmill cones.
  The lattice $\Lambda$
  has two black windmill bases given by $(-1,2),(5,3)$ and by $(-1,2),(6,1)$.

  The basis $(-1,2),(-6,-1)$ is not a black windmill basis: $(-6,-1)$
  does not belong to the upper half-plane.
  
  The basis $(5,3),(6,1)$ is not a windmill basis: Both basis vectors
  are contained in the open black E-NE windmill cone.

  The basis $(-1,2),(4,5)$ is not a windmill basis: $(4,5)$ belongs to the open
  white N-NE windmill cone and $(-1,2)$ belongs to the open black N-NW windmill
  cone.

\begin{lem}\label{lemonecolour} All windmill bases of a lattice have
  the same colour.
\end{lem}

\begin{proof} Otherwise we get a contradiction with Lemma \ref{lemnoncrossing}
  since windmill bases of different colours are obviously interlaced.
\end{proof}

An odd prime number $p$ and an element $\mu$ of $\mathbb F_p$ (henceforth
often identified with $\{0,\ldots,p-1\}$)
define a sub-lattice 
\begin{align}\label{deflambdamu}
  \Lambda_\mu(p)=\{(x,y)\in\mathbb Z,\ \vert\ x+\mu y\equiv 0\pmod p\}
\end{align}
of index $p$ in $\mathbb Z^2$. We set $\Lambda_\infty(p)=\{(x,y)\in\mathbb Z,\ \vert\ y\equiv 0\pmod p\}$. We have thus $\Lambda=\Lambda_7(13)$ for our
running example given by (\ref{runningexple}). All $p+1$ lattices $\Lambda_0(p),\ldots,
\Lambda_{p-1}(p),\Lambda_\infty(p)$ are distinct and $\mathbb Z^2$ contains
no other sublattices of prime index $p$, see Proposition \ref{propsubl}.

\begin{prop}\label{propnobasis} The four lattices $\Lambda_0(p),\Lambda_\infty(p),\Lambda_1(p),\Lambda_{p-1}(p)$ have no windmill basis.
\end{prop}

\begin{proof} Each of these four lattices is invariant under an orthogonal
reflection with respect to a line
separating black and white windmill cones.
Such orthogonal reflections, followed by obvious sign changes,
exchange white and black windmill bases.
Lemma \ref{lemonecolour} 
shows therefore non-existence of (black or white) 
windmill bases for these lattices.
\end{proof}

Proposition \ref{propnobasis} is optimal by the following result:

\begin{prop}\label{propexistencebasis} Every lattice $\Lambda_\mu(p)$
  with $2\leq \mu\leq p-2$ has a windmill basis.
\end{prop}

\begin{proof} $\Lambda_\mu(p)$ contains obviously no elements of the form
  $(\pm m,0)$ or $(\pm m,\pm p)$ with $m$ in $\{1,2,\ldots,p-1\}$.
  Since $p$ is prime, $\Lambda_\mu(p)$ contains no elements of
  the form $(0,\pm m),(\pm p,\pm m)$ with $m$ in $\{1,\ldots,p-1\}$.
  Moreover, for $\mu$ in $\{2,\ldots,p-2\}$ considered as a subset of the finite
  field $\mathbb F_p$, the elements $1+\mu$ and $1-\mu$ are invertible in $\mathbb F_p$. This implies that $\Lambda_\mu(p)$ has
  no elements of the form $\pm(m,m),\pm(m,-m)$ for $m$ in $\{1,\ldots,p-1\}$.
  The intersection of a (black or white) windmill cone with $[-p,p]^2$
  defines therefore a triangle of area $p^2/2>p/2$ whose boundary contains no
  lattice-points of $\Lambda_\mu(p)$ except for its three vertices.
  Lemma \ref{lemlatticebasis} implies now that
  every open (black or white) windmill cone contains a non-zero element $(x,y)$
  of $\Lambda_\mu(p)$
  with coordinates $x,y$ in $\{\pm 1,\pm 2,\ldots,\pm(p-1)\}$.
  
  Thus there exists a parallelogram $\mathcal P$ of minimal area with vertices
  $\pm e,\pm f$ in $\Lambda_\mu(p)\cap\{-p+1,\ldots,p-1\}^2$
  such that $\{\pm e,\pm f\}$ intersects either
  all four open black windmill cones
  or all four open white windmill cones.
  
  Suppose for simplicity that all elements of $\{\pm e,\pm f\}$
  are black (i.e. belong to open black windmill cones). (The case where
  both pairs  $\pm e$ and $\pm f$ are white is analogous.)

  Since $\Lambda_\mu(p)$
  intersects the diagonal $\mathbb R(1,1)$ and the antidiagonal $\mathbb R(1,-1)$
  in $\mathbb Z(p,p)$ and in $\mathbb Z(p,-p)$, and since
  $\Lambda_\mu(p)$ contains obviously no elements of the form $(\pm a,0),
  (0,\pm a)$ with $a$ in $\{1,\ldots,p-1\}$, all non-zero elements
  of $\mathcal P\cap \Lambda_\mu(p)$ belong to open windmill cones.
  Suppose that $\mathcal P\setminus\{\pm e,\pm f\}$ contains a non-zero
  element $g$ of $\Lambda_\mu(p)$.
  Area-minimality of $\mathcal P$ and the 
absence of non-zero elements in $\Lambda_\mu(p)\cap (\mathbb Z(1,1)\cup
\mathbb Z(1,-1))\cap\{-p+1,\ldots,p-1\}^2$
shows that 
  $g$ is contained in a white windmill cone (under the assumption
that $e$ and $f$ are black). After replacing $g$ by $-g$ if necessary,
  the element $g$ belongs either to the triangle with vertices
  $(0,0),e,f$ or to the triangle with vertices $(0,0),e,-f$.
  Lemma \ref{lemlatticebasis} applied  to the two sets $e,f$ and $e,-f$
generating the same sublattice $\mathbb Z e+\mathbb Z f$ of $\Lambda_\mu(p)$
  implies thus the existence of a non-zero element
  $h$ in $\mathcal P\cap \Lambda_\mu(p)$ such that
  $\{\pm g,\pm h\}$ intersects all four open white windmill cones.
  The parallelogram with vertices $\pm g,\pm h$ in all four open
  white windmill cones is therefore strictly included in $\mathcal P$
  in contradiction with area-minimality 
  of $\mathcal P$.

  After replacing each of $g$ and $f$ by its negative if necessary, we get
  that $e,f$ is a windmill basis of $\Lambda_\mu(p)$ by
  Lemma \ref{lemlatticebasis}.  
  \end{proof}

  \begin{lem}\label{lemuniquemin}
    Let $\Lambda$ be a lattice with two distinct windmill bases
    $e,f$ and $e,g$ sharing a common element $e$. Then $\Lambda$
    has a unique pair of minimal elements given by $\pm e$.
  \end{lem}
  
  \begin{proof} Since both ordered bases $e,f$ and $e,g$ start with $e$ and are windmill
    bases, they induce the same orientation and we have $g=f+ke$ for
    some non-zero integer $k$ in $\mathbb Z$.
    The affine line $\mathcal L=f+\mathbb R e$ intersects therefore the
    open windmill cone $\mathcal C_f$ containing $f$ and $g$ in an open segment
    of length $l>\sqrt{\langle e,e\rangle}$. Denoting by $d$ the distance
      of $\mathcal L$ to the origin $(0,0)$ and by $\alpha$
      the angle in $(0,\pi/4)$ between the normal line $(\mathbb R e)^\perp$
      (with $(\mathbb R e)^\perp\setminus\{(0,0)\}$ contained in
      $\mathcal C_f\cup(-\mathcal C_f)$) of $\mathcal L$
      and a boundary line of $\mathcal C_f$ (separating $\mathcal C_f$
      from a windmill cone of the opposite colour) we have
the inequalities
\begin{align*}
  \sqrt{\langle e,e\rangle}&<l\\
  &=d\left(\tan\alpha+\tan(\pi/4-\alpha)\right)\\
   &=(1-\tan\alpha\tan(\pi/4-\alpha))d\tan(\pi/4)\\
   &<d
\end{align*}
where we have used the addition formula $\tan(x+y)=\frac{\tan x+\tan y}
{1-\tan x\tan y}$ of the tangent function.

The open strip delimited by the two parallel affine lines $\mathcal L$ and
$-\mathcal L$ and consisting of all points at distance $<d$ from $\mathbb Re$
intersects $\Lambda$ in $\mathbb Ze$.
All elements of $\Lambda \setminus \mathbb Z e$ are therefore
at least at distance $d>\sqrt{\langle e,e\rangle}$ from the origin.
This shows that $\pm e$ is the unique pair of minimal vectors in $\Lambda$.
  \end{proof}
  
  \begin{lem}\label{lemnodisj}
    Two windmill bases of a lattice are never disjoint.
  \end{lem}

  \begin{proof} Up to replacing a lattice $\Lambda$ having
    two disjoint windmill bases
    by its orthogonal reflection $\sigma(\Lambda)$
    with respect to the vertical line $x=0$, we can assume otherwise that
    $\Lambda$ has
    two disjoint black windmill bases (see Lemma \ref{lemonecolour})
    given
    by $e,f$ and $g,h$ with $e,g$ in the open black E-NE windmill cone
    and $f,h$ in the open black N-NW windmill cone.
    Since the two bases are not interlaced by 
    Lemma \ref{lemnoncrossing}, we can moreover assume that $g$ and $h$ belong
    both to the open cone spanned by $e,f$.
     If the sum $e+f$ belongs to the open 
cone spanned by $g$ and $h$ we get two interlaced bases $e,e+f$ and 
$g,h$ in contradiction with Lemma \ref{lemnoncrossing}.

The non-zero lattice 
element $e+f$ belongs therefore either to the closed cone spanned by $e,g$ 
or to the closed cone spanned by $f,h$. In the first case ($e+f$
in the closed cone spanned by $e,g$) the element
$e+f$ belongs to the open black E-NE windmill cone containing $f$ and $g$,
see Figure \ref{figpreuveI} for a schematic illustration.
(We discuss here only the first case.
The second case where $e+f$ belongs to the closed cone spanned by $f$
and $h$ reduces to the first cases after a quarter-turn of the lattice
$\Lambda$ followed by obvious relabellings and sign-changes.
It can also be treated by an easy adaption of the
arguments used in the first case.) 
\begin{figure}[h]
  \epsfysize=5.4cm
  \center{\epsfbox{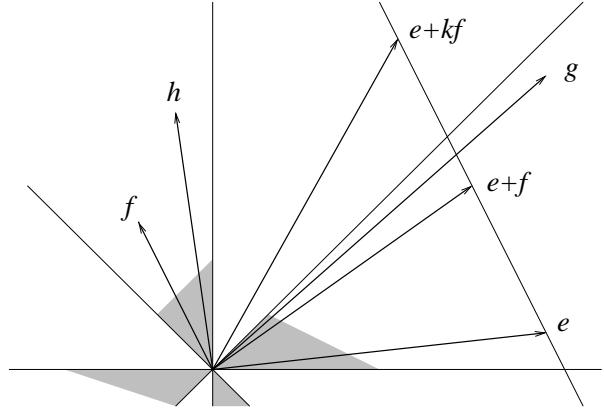}}
\caption{A schematic figure with $e+f$ in the E-NE windmill cone.}\label{figpreuveI}
\end{figure}
Since the affine line $\mathcal L=e+\mathbb R f$
has a downward slope strictly steeper than $-1$,
the intersection 
of $\mathcal L$ with the open white N-NE windmill cone is strictly longer
than the intersection of $\mathcal L$ with the open black E-NE windmill cone.
Since the intersection of $\mathcal L$ with the open black E-NE windmill
cone contains at least the two elements $e$ and $e+f$ of $\Lambda$
there exists a natural integer $k>1$ such that the element $e+kf$ of 
$\Lambda$ belongs to the open white N-NE windmill cone.
This leads to two interlaced bases $f,e+kf$ and $g,h$ in contradiction 
with Lemma \ref{lemnoncrossing}.
    \end{proof}

  \begin{prop}\label{propcommonelt}
    The following assertions hold if a lattice
    $\Lambda$ has at least two windmill bases:

    \begin{itemize}
    \item{} $\Lambda$ has a unique pair $\pm m$ of minimal vectors
      with $m$ contained in all windmill bases of $\Lambda$.

    \item{} $\Lambda$ has only finitely many windmill bases.
    More precisely, there exists a windmill basis $m,f$ of $\Lambda$ such that
      all windmill bases of $\Lambda$ are given by $m,f+sm$ for
      $s$ in $\{0,1,\ldots,k-1\}$ where $k$ is the number of windmill
      bases of $\Lambda$.
    \end{itemize}    
  \end{prop}  

  Minimal vectors do not necessarily intersect the set
  $\{e,f\}$ in the case of a lattice $\Lambda$ with a unique windmill basis
  $e,f$.
  
  Both black windmill bases $(-1,2),(5,3)$ and $(-1,2),(6,1)$ of our
  running example $\Lambda$ defined by
  (\ref{runningexple}) contain the minimal element $m=(-1,2)$ of $\Lambda$.
  The running example gives $f=(6,1)$ and $k=2$ in the last assertion.

  \begin{proof}[Proof of Proposition \ref{propcommonelt}]
    Lemma \ref{lemnodisj} shows that two windmill bases of
    $\Lambda$ intersect in a common element
    $m$ defining the unique pair $\pm m$ of  minimal vectors of $\Lambda$
    by Lemma \ref{lemuniquemin}.
    Thus all windmill bases of $\Lambda$ are all pairs $m,g$
    with $g$ in the set $(f+\mathbb R m)\cap \mathcal C_m^\perp\cap\Lambda$
    where $m,f$ is an arbitrary  windmill basis and where
    $\mathcal C_m^\perp$ is the open windmill cone containing $f$
    perpendicular to the open windmill cone $\mathcal C_m$
    defined by $m$.

Since the vector $m$ does not
belong to $\mathcal C_m^\perp$ the intersection
of the affine line $\mathcal L=f+\mathbb R m$
    with $\mathcal C_m^\perp$ is an open interval of bounded length and
    the set $\mathcal G=\mathcal L\cap \mathcal C_m^\perp\cap\Lambda$
    is finite. We replace now $f$ by the element $g$ of $\mathcal G$ minimizing the
    scalar product with $m$ in order to get the result.
\end{proof}
    
We call a black windmill basis $u,v$ of a lattice $\Lambda_\mu(p)$
(with $\mu$ in $\{2,\ldots,p-2\}$) \emph{standard} if $u=(a,c),v=(-d,b)$
with $a,b,c,d\in\mathbb N$ such that $\min(a,b)>\max(c,d)$.

The basis $u=(6,1),v=(-1,2)$ of our running example (\ref{runningexple})
is a standard black windmill basis. The inequality $3\geq 2$ implies
that its second black windmill basis
$(5,3),(-1,2)$ is not standard.

\begin{prop}\label{propmonochr} Given an odd prime number $p$, a lattice $\Lambda_\mu(p)$ 
with $\mu$ in $\{2,\ldots,p-2\}$ has either only white windmill bases
or it has a unique standard black windmill basis.
\end{prop}

\begin{proof} Proposition \ref{propexistencebasis} shows that
  such a lattice $\Lambda=\Lambda_\mu(p)$ has windmill bases.
  They are all of the same colour
  by Lemma \ref{lemonecolour}. We assume now that all windmill bases
  of $\Lambda$ are black. A vector $w$ in $\Lambda$ is of \emph{windmill type}
  if there exists a windmill basis containing $w$.

  We denote by $u=(a,c)\in\Lambda$ the lowest vector of windmill type in
  the open black E-NE windmill cone and we denote by $v=(-d,b)\in\Lambda$
  the rightmost vector of windmill type in the open black N-NE windmill cone.
  The vectors $u,v$ form a black windmill basis of $\Lambda$: This is obvious
  if $\Lambda$ has a unique windmill basis and it follows from the
  description of all windmill bases given by
  Proposition \ref{propcommonelt} otherwise.

  For our running example (\ref{runningexple})  we get $u=(6,1)$ and $v=(-1,2)$.
 
 We claim that $u,v$ is a standard black windmill basis of $\Lambda$:
 We have $a>c$ since $u=(a,c)$ belongs to the open black E-NE windmill cone
 and $b>d$ since $v=(-d,b)$ belongs to the open black N-NW windmill cone.

 Since $u-v=(a+d,c-b)$ is lower than $u$, the basis $u-v=(a+d,c-b),v=(-d,b)$
 of $\Lambda$ is not a windmill basis and we have therefore
 $b\geq c$.
 If $b=c$, the vectors $u-v=(a+d,0),v=(-d,b)$ are a basis of $\Lambda$.
 Since $\Lambda$ intersects $\mathbb Z(1,0)$ in $\mathbb Z(p,0)$
 we have $a+d=p$ which implies $u-v=(p,0)$.
 Since $u-v=(p,0)$ and $v=(-d,b)$ is a basis of $\Lambda$
 we get $b=1$ in contradiction with the inequalities $1\leq d<b$.
 We have thus $b>c$.

 Similarly, since $v+u=(a-d,b+c)$ is at the right of $v$, the basis
 $u=(a,c),v+u=(a-d,b+c)$ of $\Lambda$ is not a windmill basis and
 we have $a\geq d$.
 If $a=d$, the vectors $u=(a,c),v+u=(0,b+c)$ are a basis of $\Lambda$.
 This implies $b+c=p$ and $a=1$ contradicting the inequalities $1\leq c<a$.
 This shows $a>d$.

 Unicity follows easily from the description of
 all windmill bases given by the last assertion of
 Proposition \ref{propcommonelt}.
 \end{proof}


\begin{proof}[Proof of Theorem \ref{thmmain}]
  Given an odd prime number $p$, we denote by
  $\mathcal S_p$ the set of all
  solutions $(a,b,c,d)$ as
  defined by Theorem \ref{thmmain}.

  We associate to a solution $(a,b,c,d)$ in $\mathcal S_p$
  the two vectors $u=(a,c),\ v=(-d,b)$
and we consider the sublattice
$\Lambda=\mathbb Zu+\mathbb Zv$ of index $p=ab-c(-d)$ in $\mathbb Z^2$ generated by $u$ and $v$.
Since $p$ is prime, there are exactly two solutions with $cd=0$, given by
$(p,1,0,0)$ and $(1,p,0,0)$ corresponding to the
lattices $\mathbb Z(p,0)+\mathbb Z(0,1)$ and
$\mathbb Z(1,0)+\mathbb Z(0,p)$.

We suppose henceforth $cd>0$.
The vectors $u$ and $v$ are then contained respectively in
the open black E-NE and in the open
black N-NW windmill cone and form therefore a standard black
windmill basis of the lattice $\Lambda$.

Sub-lattices of prime-index $p$ in $\mathbb Z^2$
are in bijection with the set of all $p+1$ points on
the projective line $\mathbb P^1\mathbb F_p$ over the finite
field $\mathbb F_p$, cf. Proposition \ref{propsubl}.
More precisely, a point $[a:b]$ of the projective line defines the
lattice
$$\Lambda_{[a:b]}=\{(x,y)\in\mathbb Z^2\ \vert\ ax+by\equiv 0\pmod p\}$$
which is equal to the lattice $\Lambda_\mu(p)$ defined by
(\ref{deflambdamu})  for $\mu\equiv b/a
\pmod p$ using obvious conventions.
We have already considered lattices associated to the two solutions
with $cd=0$. By Proposition \ref{propnobasis}, the two lattices given by
$\mu\equiv\pm 1\pmod p$
have no windmill basis and yield thus no solutions.
All $(p-3)$ lattices $\Lambda_\mu(p)$ with $\mu\in\{2,\ldots,p-2\}$
have windmill bases by Proposition \ref{propexistencebasis}.

Since $\Lambda_\mu(p)$ and $\Lambda_{p-\mu}(p)$
  (respectively $\Lambda_{\mu^{-1}\pmod p}(p)$) differ by a horizontal
(respectively diagonal) reflection,
they have windmill bases of different colours.
Proposition \ref{propmonochr} shows that exactly
$(p-3)/2$ values of $\mu$ in $\{2,\ldots,p-2\}$
correspond to lattices $\Lambda_\mu(p)$ with unique standard black
windmill bases. These $(p-3)/2$ lattice are therefore
in one-to-one correspondence  
with solutions in $(a,b,c,d)$ in $\mathcal S_p$ such that $cd>0$. Taking into
account the two degenerate solutions $p\cdot 1+0\cdot 0$ and
$1\cdot p+0\cdot 0$, we get a total number of $(p-3)/2+2=(p+1)/2$
solutions in $\mathcal S_p$.
\end{proof}


\section{Complement: Voronoi cells and windmill bases}

A lattice $\Lambda$ induces a \emph{Voronoi} diagram tiling the plane
$\mathbb R^2$ with \emph{Voronoi cell} bounded by points of
$\mathbb R^2$ having more than a unique closest lattice point.
Points at locally maximal distance to $\Lambda$
are vertices of the Voronoi diagram for $\Lambda$, see for example
the monograph \cite{CS} for more on Voronoi cells of lattices and
sphere packings.

We denote by $\mathcal C_V$ the Voronoi cell consisting of all points
of $\mathbb R^2$ with closest lattice-point
the trivial element $(0,0)$ of $\Lambda$. 
The plane $\mathbb R^2$ is tiled by $\Lambda$-translates of
$\mathcal C_V$ which is a rectangle if $\Lambda$ has a reduced basis of two
orthogonal vectors and which is a centrally symmetric
hexagon otherwise. A reduced basis $e,f$ defines
normal vectors to two pairs of parallel
sides of $\mathcal C_V$. A normal vector $g$
to the remaining pair of sides in the hexagonal case is given by
$e-\epsilon f$
where $\epsilon=\frac{\langle e,f\rangle}{\vert\langle e,f\rangle\vert}$
is the sign of the scalar product $\langle e,f\rangle$ between
$e$ and $f$. We call the set $\{\pm e,\pm f\}$, respectively $\{\pm e,
\pm f,\pm g\}$ of primitive lattice elements normal to sides of $\mathcal C_V$
the set of \emph{Voronoi vectors}. The Voronoi cell $\mathcal C_V$ is defined
by the inequalities $2\langle x,v\rangle\leq \langle v,v\rangle$ for
all elements
$v$ of the set $\mathcal V$ of Voronoi vectors. Two
linearly independent Voronoi vectors $u,v$ in $\mathcal V$ generate $\Lambda$.

We illustrate these notions on our running example
$\Lambda=\mathbb Z(-1,2)+\mathbb Z(6,1)$ defined by 
(\ref{figrunexple}): A reduced basis for $\Lambda$ is given by the
minimal vector $e=(-1,2)$ and by $f=(5,3)$. Voronoi vectors are given by
$\{\pm e,\pm f,\pm g\}$ for $g=f-e=(6,1)$. The lattice $\Lambda$ has
two black windmill bases given by $\{e,f\}$ and $\{e,g\}$. The last basis
$\{e,g\}$ is the standard basis of $\Lambda$. The Voronoi cell $\mathcal C_V$
of $\Lambda$ is the hexagon with vertices $\pm(53/26,59/26),\pm (77/26,19/26),
\pm(79/26,7/26)$.

Voronoi vectors are related to windmill bases by the following result:

\begin{prop}\label{propVoron} Every lattice $\Lambda$ in $\mathbb R^2$ with windmill bases
has a windmill basis contained in its set of Voronoi vectors.
\end{prop}

Proposition \ref{propVoron} (together with Proposition \ref{propcommonelt}
and Gau\ss ian lattice reduction) gives a fast algorithm (using $O(\log p)$
operations on integers not exceeding $p$) for
computing the solution of $\mathcal S_p$
associated to $\Lambda_{\pm \mu}(p)$ for $\mu$ in $\mathbb F_p\setminus\{0,\pm 1\}$:
Compute a reduced basis of $\Lambda_\mu(p)$ and use it to construct the associated set $\mathcal V$ of Voronoi vectors
which contains a windmill basis by Proposition
\ref{propexistencebasis} and Proposition \ref{propVoron}. If
the windmill basis is white, replace $\Lambda_{\mu}(p)$ by $\Lambda_{-\mu}(p)$
(using for example the vertical reflection
$\sigma(x,y)=(-x,y)$ of $\mathbb R^2$) in order to get a
black windmill basis of $\Lambda_{-\mu}(p)$.
Use now Proposition \ref{propcommonelt}
for constructing the unique standard basis $(a,c),(-d,b)$
encoding the solution $p=ab+cd$ of $\mathcal S_p$.

The main tool for proving Proposition \ref{propVoron} is the following 
result which is perhaps of independent interest:

\begin{lem}\label{lemgenpos} Assume that the
  set $\mathcal V$ of Voronoi vectors of a lattice $\Lambda$
  does not intersect the set of boundary lines separating black and white 
windmill cones. Then $\Lambda$ has a windmill basis contained in $\mathcal V$.
\end{lem}

\begin{proof}[Proof of Lemma \ref{lemgenpos}] We suppose first that 
the Voronoi domain of $\Lambda$ is a rectangle. This implies
$\mathcal V=\{\pm e,\pm f\}$ with $e$ and $f$ two elements of open windmill
cones in the upper halfplane forming a reduced orthogonal 
basis of $\Lambda$. Since $e$ and $f$ are orthogonal they belong to two
distinct open windmill cones of the same colour.
They form therefore a windmill basis.

We consider now $\mathcal V=\{\pm e,\pm f,\pm g\}$
with $e,f,g$ in open windmill cones of the
  upper half-plane. We assume moreover that $e$ is a (perhaps not unique)
  minimal vector of $\Lambda$. Minimality of $e$ and the inequalities 
$\vert\langle e,f\rangle\vert< \langle e,e\rangle$ and
$\vert \langle e,g\rangle\vert< \langle e,e\rangle$
imply that the line
$\mathbb R e$ crosses both lines $\mathbb R f$ and $\mathbb R g$
with angles strictly larger than $\pi/4$. The open
windmill cone $\mathcal C_e$ containing $e$ has an opening angle of $\pi/4$
and is therefore distinct
from the (not necessarily distinct) open windmill cones $\mathcal C_f$
and $\mathcal C_g$ containing $f$, respectively $g$. We get
a windmill
basis $e,h$ for $h$ in $\{f,g\}$ such that $\mathcal C_e$ and $\mathcal C_h$
are of the same colour. If such an element $h$ does not exist, then
$\mathcal C_f$ and $\mathcal C_h$ have the same colour opposite to the
colour of $\mathcal C_e$.
Since $f$ and $g$ are either separated
by the line $\mathbb Re$ or by its orthogonal $(\mathbb R e)^\perp$
(with $(\mathbb R e)^\perp\setminus\{(0,0)\}$
contained in the two open windmill cones orthogonal to $\mathcal C_e$ and
of the same colour as $\mathcal C_e$),
the elements $f$ and $g$ of the upper half-plane
belong to different open windmill cones of the same colour
and form therefore a windmill basis.
\end{proof}

\begin{proof}[Proof of Proposition \ref{propVoron}] The result holds
  by Lemma \ref{lemgenpos} if $\mathcal V$ contains no elements
  on boundary lines separating black and white windmill cones.

  Otherwise, if $\mathcal V=\{\pm e,\pm f\}$ is reduced to two
  pairs of
  orthogonal elements, then $e,f$ are both elements in the boundary of black
  and white windmill cones. The reflection $ae+bf\longmapsto ae-bf$ induces
  therefore a lattice isomorphism of $\Lambda$ which exchanges colours of
  windmill cones. Such a lattice has no windmill bases
  by Lemma \ref{lemonecolour}.

  We suppose now that $\Lambda$ has a windmill basis $u,v$ not contained
  in the set $\mathcal V=\{\pm e,\pm f,\pm g\}$ of Voronoi vectors,
  with $e,f,g$ in the closed upper half-plane.
  A sufficiently small rotation $\rho$ (suitably chosen if
$\{e,f,g\}$ intersects the horizontal line $y=0$) 
  sends the windmill basis $u,v$ of $\Lambda$
to a windmill basis (of the same colour) $\rho(u),\rho(v)$ of $\rho(\Lambda)$
and sends $\mathcal V$ to a set $\rho(\mathcal V)$ of Voronoi
vectors having no elements on boundary
lines separating windmill cones. Lemma \ref{lemgenpos} shows
that $\rho(\Lambda)$ has an additional windmill basis contained in $\rho(\mathcal V)$ distinct from the windmill basis $\rho(u),\rho(v)$ not contained
in $\rho(\mathcal V)$.
Proposition \ref{propcommonelt} implies therefore
that $\rho(\Lambda)$
has a unique pair $\pm \rho(m)$ of minimal vectors
intersecting every windmill basis of
$\rho(\Lambda)$. We can therefore assume (perhaps after a permutation
among the elements $e,f,g$) that $e$ is the unique minimal element in
the upper half-plane of $\Lambda$
and that $e$ is contained in every windmill basis of
$\Lambda$.
Since $\pm f$ and $\pm g$ are the elements of $\pm f+\mathbb Ze$ which are
closest to the line $(\mathbb R e)^\perp$ orthogonal to $\mathbb R e$, the
last assertion of Proposition \ref{propcommonelt}
implies that either $e,f$ or $e,g$ is a windmill basis of $\Lambda$.
\end{proof}

{\bf Acknowledgements:} I thank M. Decauwert, P. Dehornoy, 
P. De la Harpe, C. Elsholtz, A. Guilloux, 
C. Leuridan, C. MacLean, E. Peyre, L. Spice
for comments, remarks or questions. Special thanks to C. Leuridan
and to two anonymous referees
whose  suggestions improved the text hugely.


\begin{thebibliography}{99}


\bibitem{C} de Cervantes, M. (1605), 
  \textit{El ingenioso hidalgo don Quijote de la Mancha}, Madrid.
    
\bibitem{Chr} Christopher, D.A. (2016),
A partition-theoretic proof of Fermat's two squares theorem.
\textit{Discrete Math.} 339(4), 1410--1411. 

\bibitem{CS} Conway, J.H.; Sloane, N.J.A. (1999),
Sphere packings, lattices and groups. 3rd ed. Springer, 703 p. 

\bibitem{Do} Dolan, S. (2016), A very simple proof of the two-squares theorem, The Mathematical Gazette, Volume 105, Issue 564, 511. 

\bibitem{Els} Elsholtz C. (2010), {\it 
A combinatorial approach to sums of two squares and related problems.} Additive number theory. Festschrift in honor of the sixtieth birthday of Melvyn B. Nathanson (edited by 
D. Chudnovsky, David et al.), Springer, 115--140 (2010). 

\bibitem{Gra} Grace, J.H. (1927), The four square theorem,
\textit{J. London Math. Soc.} 2, 3--8.

\bibitem{HW} Hardy, G.H.; Wright, E.M. (1960), 
\textit{An introduction to the theory of numbers.} Fourth ed. 
Oxford University Press, 421 p. 

\bibitem{HB} D.R. Heath-Brown (copyright 1984), Fermat's Two Squares Theorem,
  archived by Oxford University: https://ora.ox.ac.uk/.

\bibitem{Mo} Mathoverflow, MO31113:   https://mathoverflow.net/questions/31113.
  
\bibitem{MO405035} Mathoverflow, MO405035: https://mathoverflow.net/questions/405035.

\bibitem{Sp} Spivak, A. : \foreignlanguage{russian}{Крылатые квадраты} (Winged squares), Lecture notes for the mathematical circle at Moscow State University, 15th lecture 2007.
  https://mmmf.msu.ru/lect/spivak/summa\_sq.pdf
  
\bibitem{WikipFerm} Wikipedia (June 2022), Fermat's theorem on sums of two squares, 
  https://en.wikipedia.org/wiki/Fermat's\_theorem\_on\_sums\_of\_two\_squares.
  
\bibitem{Za} Zagier, D. (1990),
A one-sentence proof that every prime $p\equiv 1\pmod 4$
is a sum of two squares.
\textit{Amer. Math. Monthly} 97(2), 144.

\end{thebibliography}
\end{document}